\documentclass[a4paper,12pt]{article}
\usepackage{graphics}
\usepackage{epsf}

\usepackage{amsfonts}
\usepackage{amssymb}
\usepackage{epsfig}
\usepackage[latin1]{inputenc}
\usepackage{color}

\usepackage{amsthm,amsmath,amssymb}

\newtheorem{theorem}{Theorem}

\newtheorem{lemma}[theorem]{Lemma}
\newtheorem{remark}[theorem]{Remark}
\newtheorem{corollary}[theorem]{Corollary}

\begin{document}
\title{Exponential quadrature rules without order reduction}
\author{
{\sc  B. Cano \thanks{Corresponding author. Email: bego@mac.uva.es} } \\ \small
IMUVA, Departamento de Matem\'atica Aplicada,\\ \small Facultad de Ciencias, Universidad de
Valladolid,\\ \small Paseo de Bel\'en 7, 47011 Valladolid,\\ \small Spain \\
{\sc and}\\
{\sc M. J. Moreta}\thanks{Email:mjesusmoreta@yahoo.es} \\
\small IMUVA,
Departamento de Fundamentos del Análisis Económico I, \\ \small
Facultad de Ciencias Económicas y Empresariales, Universidad
Complutense de Madrid, \\[2pt] \small Campus de Somosaguas, Pozuelo de Alarcón,
28223 Madrid, \\ \small Spain.
}
\date{}

\maketitle

\begin{abstract}
In this paper a technique is suggested to integrate linear initial boundary value problems with exponential quadrature rules in such a way that the order in time is as high as possible. A thorough error analysis is given for both the classical approach of integrating the problem firstly in space and then in time and of doing it in the reverse order in a suitable manner. Time-dependent boundary conditions are considered with both approaches and full discretization formulas are given to implement the methods once the quadrature nodes have been chosen for the time integration and a particular (although very general) scheme is selected for the space discretization. Numerical experiments are shown which corroborate that, for example, with the suggested technique, order $2s$ is obtained when choosing the $s$ nodes of  Gaussian quadrature rule.

\end{abstract}



\section{Introduction}
Due to the recent development and improvement of Krylov methods \cite{grimm, niesen}, exponential quadrature rules have become a valuable tool to integrate linear initial value partial differential problems \cite{HO2}. This is because of the fact that the linear and stiff part of the problem can be \lq exactly' integrated in an efficient way through exponential operators. Moreover, when the source term is nontrivial, a variations-of-constants formula and the interpolation of that source term in several nodes leads to the appearance of some $\varphi_j$-operators, to which Krylov techniques can also be applied.

However, in the literature \cite{HO2}, these methods have always been applied and analysed either on the initial value problem or on the initial boundary value one with vanishing or periodic boundary conditions. More precisely, when the linear and stiff operator is the infinitesimal generator of a strongly continuous semigroup in a certain Banach space. In such a case, it has been proved \cite{HO2} that the exponential quadrature rule converges with global order $s$ if $s$ is the number of nodes being used for the interpolation of the source term.

There are no results concerning specifically how to deal with these methods when integrating the most common non-vanishing and time-dependent boundary conditions case. The only related reference is that of Lawson quadrature rules \cite{acr,acr2,lawson} which differ from these methods in the fact that, not only the source term is interpolated, but all the integrand which turns up in the variations-of-constants formula. In such a way, with the latter methods, $\{\varphi_j\}$-operators (with $j\ge 1$) do not turn up. Only exponential functions (those corresponding to $j=0$) are present. In \cite{acr} a thorough error analysis is given which studies the strong order reduction which turns up with Lawson methods even in the vanishing boundary conditions case unless some even more artificial additional vanishing boundary conditions are satisfied. Nevertheless, in \cite{acr2}, a technique is suggested to avoid that order reduction under homogeneous boundary conditions and to even tackle the time-dependent boundary case without order reduction. Moreover, the analysis there also includes the error coming from the space discretization.

The aim of this paper is to generalize that technique to the most common quadrature rules which are used in the literature and which also use $\{\varphi_j\}$-operators. Besides, we also include the error coming from the space discretization not only when avoiding order reduction but also with the classical approach. We will see that, with the technique which is suggested in this paper, we manage to get the  order of the classical quadrature interpolatory rule, which implies that, by choosing the $s$ nodes carefully, we can manage to get even order $2s$ in time.

The paper is structured as follows. Section 2 gives some preliminaries on the abstract framework in Banach spaces which is used for the time integration of the problem, on the definition of the exponential quadrature rules and on the general hypotheses which are used for the abstract space discretization. Then, Section 3 describes and makes a thorough error analysis of the classical method of lines, which integrates the problem firstly in space and then in time. Both vanishing and nonvanishing boundary conditions are considered there and several different results are obtained depending on the specific accuracy of the quadrature rule and on whether a parabolic assumption is satisfied. After that, the technique which is suggested in the paper to improve the order of accuracy in time is well described in Section 4. It consists of discretizing firstly in time with suitable boundary conditions and then in space. Therefore, the analysis is firstly performed on the local error of the time semidiscretization and then on the local and global error of the full scheme (\ref{Vh0}),(\ref{Vhij}),(\ref{ff}).
Finally, Section 5 shows some numerical results which corroborate the theoretical results of the previous sections.

\section{Preliminaries}
\label{prel}
As in \cite{acr2},  we consider the linear abstract initial boundary value problem in the complex Banach space $X$
\begin{eqnarray}
\label{laibvp}
\begin{array}{rcl}
u'(t)&=&Au(t)+f(t),\quad 0\le t \le T,\\
u(0)&=&u_0,\\
\partial u(t)&=&g(t),
\end{array}
\end{eqnarray}
where $D(A)$ is a dense subspace of $X$ and the linear operators $A$ and $\partial$ satisfy the following assumptions:
\begin{enumerate}
\item[(A1)] The boundary operator $\partial:D(A)\subset X\to Y$ is
onto, where $Y$ is another Banach space.

 \item[(A2)] Ker($\partial$) is dense in $X$ and
$A_0:D(A_0)=Ker(\partial)\subset X \to X$, the restriction of $A$
to Ker($\partial$), is the infinitesimal generator of a $C_0$-
semigroup $\{e^{tA_0}\}_{t\ge 0}$ in $X$, which type is denoted by
$\omega$, and we assume that it is negative.
\item[(A3)] If $z \in \mathbb{C}$ satisfies $ Re (z) >
\omega$ and $v \in Y$, then the steady state problem
\begin{eqnarray}
 Ax &=& zx, \nonumber\\
 \partial x&=&v,
\nonumber
\end{eqnarray}
possesses a unique solution denoted by $x=K(z)v$. Moreover, the linear operator
$K(z): Y \to D(A)$ satisfies
\begin{eqnarray}
 \| K(z)v\|_X \le  L\|v\|_Y,
\nonumber
\end{eqnarray}
where the constant $L$ holds for any $z$ such that $Re (z) >
\omega_0 > \omega$.
\end{enumerate}

As precisely stated in \cite{acr2,arti1}, with these assumptions, problem (\ref{laibvp}) is well-posed
in $BV/L^{\infty}$ sense. Moreover, as $\omega$ is negative,  the semigroup $e^{t A_0}$ decays exponentially when $t \to 0$ and the operator $A_0$ is invertible.

We also remind that, because of hypothesis (A2), $\{\varphi_j(tA_0)\}_{j=0}^{\infty}$ are
bounded operators for $t>0$, where
$\{\varphi_j\}$ are the standard functions which are used in exponential
methods \cite{HO2}:
\begin{eqnarray}
\varphi_j( t A_0)=\frac{1}{t^j} \int_0^t
e^{(t-\tau)A_0}\frac{\tau^{j-1}}{(j-1)!}d\tau. \label{varphi}
\end{eqnarray}
It is well-known that they can be calculated in a recursive way through the formula
\begin{eqnarray}
\varphi_{k+1}(z)=\frac{\varphi_k(z)-1/k!}{z}, \quad z \neq 0,
\qquad \varphi_{k+1}(0)=\frac{1}{(k+1)!}, \qquad \varphi_0(z)=e^z,
\label{recurf}
\end{eqnarray}
and that these functions are  bounded on the complex plane when
$\mbox{Re}(z)\le 0$.

We also assume that the solution of (\ref{laibvp}) satisfies that, for a natural number $p$,
\begin{eqnarray}
A^{p+1-j} u^{(j)} \in C([0,T],X), \quad 0\le j\le p+1.
\label{reg}
\end{eqnarray}
When $A$ is a differential space operator, this assumption implies that the time derivatives of the solution
are regular in space, but without imposing any restriction on the
boundary. Because of Theorem 3.1 in \cite{isaias0}, this assumption is
satisfied when the data $u_0$, $f$ and $g$ are regular and satisfy
certain natural compatibility conditions at the boundary. Moreover, as from
(\ref{laibvp}),
\begin{eqnarray}
u^{(j)}(t)=\sum_{l=0}^{j-1}
A^l f^{(j-1-l)}(t)+A^j u(t), \quad 0\le j \le p+1,\label{ujt}
\end{eqnarray}
the following crucial boundary values
\begin{eqnarray}
\partial A^j u(t)= g^{(j)}(t)-\sum_{l=0}^{j-1} \partial
A^l f^{(j-1-l)}(t), \quad 0\le j \le p+1,
\nonumber
\end{eqnarray}
can be calculated from the given data in problem (\ref{laibvp}).

We will center on  exponential quadrature rules
\cite{HO2} to time integrate (\ref{laibvp}).  When applied to a finite-dimensional linear
problem like
\begin{eqnarray}
U'(t) = B U(t)+F(t), \label{linfd}
\end{eqnarray}
where $B$ is a matrix, these rules
correspond  to interpolating $F$ in $s$ nodes $\{c_i\}_{i=1}^s$
in the integral in
the equality
\begin{eqnarray}
 U(t_{n+1})=e^{k B} U(t_n)+k\int_0^1
e^{k(1-\theta)B}F(t_n+\theta k)d\theta,
\label{vcf}
\end{eqnarray}
which is satisfied by the solutions of (\ref{linfd}) when
$t_{n+1}=t_n+k$. This yields
\begin{eqnarray}
U^{n+1}=e^{k B} U^n+ k \sum_{i=1}^s b_i(k B) F(t_n+ c_i k), \nonumber
\end{eqnarray}
with weights
\begin{eqnarray}
b_i(k B)=\int_0^1 e^{k(1-\theta)B} l_i(\theta) d \theta,
\nonumber
\end{eqnarray}
where $l_i$ are the Lagrange interpolation polynomials corresponding to the nodes $\{c_i\}_{i=1}^s$. We will define the values $\{a_{ij}\}_{i,j=1}^s$ in such a way that
\begin{eqnarray}
l_i(\theta)=a_{i,1}+a_{i,2}\theta+a_{i,3}\frac{\theta^2}{2}+\dots+a_{i,s}\frac{\theta^{s-1}}{(s-1)!}. \label{litheta}
\end{eqnarray}
From this,
$$
b_i(\tau B)=\int_0^1 e^{\tau(1-\theta)B} l_i(\theta) d \theta=\frac{1}{\tau}\int_0^{\tau} e^{(\tau-\sigma)B} l_i(\frac{\sigma}{\tau})d\sigma=\sum_{j=1}^s a_{i,j} \varphi_j(\tau B),
$$
for the functions $\varphi_j$ in (\ref{varphi}), and the final formula for the integration of (\ref{linfd}) is
\begin{eqnarray}
U^{n+1}=e^{k B} U^n+ k \sum_{i,j=1}^s a_{i,j}\varphi_j(k B) F(t_n+ c_i k). \label{qr}
\end{eqnarray}

We will consider an abstract spatial discretization  which satisfies the same hypotheses as in \cite{acr2} (Section 4.1) and which includes a big range of techniques. In such a way, for each parameter $h$ in a sequence $\{ h_j\}_{j=1}^{\infty}$ such that $h_j\to 0$, $X_h\subset X$ is a finite dimensional space which approximates $X$ when $h_j\to 0$ and the elements in $D(A_0)$ are approximated in a subspace $X_{h,0}$. The norm in $X_h$ is denoted by $\|\cdot \|_h$. The operator $A$ is then approximated by $A_h$, $A_0$ by $A_{h,0}$ and the solution of the elliptic problem
$$ A w =F, \qquad \partial w=g,$$
is approximated by $R_h w+Q_h g$, where $R_h w\in X_{h,0}$ is called the elliptic projection, $Q_h g\in X_h$ discretizes the boundary values and the following is satisfied:
\begin{eqnarray}
A_{h,0} R_h w+A_h Q_h g= L_h F,
\label{elipp}
\end{eqnarray}
for a projection operator $L_h:X\to X_{h,0}$. We will also use $P_h=L_h -L_h Q_h \partial$ and we remind part of hypothesis (H3) in \cite{acr2}, which  states that, for a subspace $Z$ of $X$ with norm $\|\cdot \|_Z$, whenever $u\in Z$,
\begin{eqnarray}
\|A_{h,0}(R_h -P_h)u\|_h \le \varepsilon_h \|u\|_Z,
\label{epsi}
\end{eqnarray}
for $\varepsilon_h$ decreasing with $h$ and, therefore, this gives a bound for the error in the space discretization of operator $A$.

Moreover, we will assume that this additional hypothesis is satisfied:
\begin{enumerate}
\item[(HS)] $\| A_{h,0}^{-1} A_h Q_h \|_h $ is bounded independently of $h$ for small enough $h$. Considering (\ref{elipp}), this in fact corresponds to a discrete maximum principle, which would be simulating the continuous maximum principle which is satisfied because of (A3) when $z=0$.
\end{enumerate}

\section{Classical approach: Discretizing firstly in space and then in time}
When considering vanishing boundary conditions in (\ref{laibvp}) (which has been classically done in the literature with exponential methods \cite{HO2}), discretizing first in space and then in time leads to the following semidiscrete problem in $X_{h,0}$:
\begin{eqnarray}
U_h'(t)&=& A_{h,0} U_h(t)+L_h f(t), \nonumber \\
U_h (0)&=& L_h u(0). \nonumber
\end{eqnarray}
When integrating this problem with an exponential quadrature rule which is based on $s$ nodes (\ref{qr}), the following scheme arises:
\begin{eqnarray}
U_h^{n+1}=e^{k A_{h,0}} U_h^n+ k \sum_{i,j=1}^s a_{i,j}\varphi_j(k A_{h,0}) L_h f(t_n+ c_i k). \label{qrvbd}
\end{eqnarray}
Denoting by $\rho_{h,n+1}$ to $U_h(t_{n+1})-\bar{U}_h^{n+1}$ where $\bar{U}_h^{n+1}$ is the result of applying (\ref{qrvbd}) from $U_h(t_n)$ instead of $U_h^n$; and $e_{h,n+1}$ to $U_h(t_{n+1})-U_h^{n+1}$ the following result follows:
\begin{theorem}
\label{teorcavbc}
Whenever $g(t)=0$ in (\ref{laibvp}), $u\in C([0,T],Z)$ and $f\in  C^{s}([0,T], X)$,
\begin{enumerate}
\item[(i)]
$\rho_{h,n}=O(k^{s+1})$,
\item[(ii)]
$e_{h,n}=O(k^s)$,
\item[(iii)]
$L_h u(t_n)- U_h^n=O(k^s+ \varepsilon_h)$.
\end{enumerate}
where the constants in Landau notation are independent of $k$ and $h$.
\end{theorem}
\begin{proof}
(i) comes from the fact that the difference between $f(t_n+k \tau)$ and its interpolant $I(f(t_n+k \tau))$ in those nodes is $O(k^s)$. More explicitly,
by using (\ref{vcf}) and the definition of $\bar{U}_h^{n+1}$,
\begin{eqnarray}
 \hspace{0.5cm}
 \rho_{h,n+1} =U_h(t_{n+1})-\bar{U}_h^{n+1}=k \int_0^1 e^{k(1-\theta)A_{h,0}} L_h [f(t_n+k \theta)-I(f(t_n+k \theta))]d\theta. \label{loccd}
\end{eqnarray}
Now, taking into account hypotheses (H1)-(H2) in \cite{acr2}, $e^{k(1-\tau)A_{h,0}}$ and $L_h$ are bounded with $h$, and the result follows.

Then, (ii) is deduced from the classical argument for the global error once the local error is bounded.
Finally, (iii) comes from (ii) and the decomposition
$$L_h u(t_n)-U_h^n= [L_h u(t_n)- U_h(t_n)]+[U_h(t_n)-U_h^n],$$
by noticing that, for the first term, as $g=0$, it happens that
\begin{eqnarray}
L_h \dot{u}(t)-\dot{U}_h(t)&=& A_{h,0}(R_h u(t)-U_h(t))=A_{h,0}(L_h u(t)-U_h(t))+A_{h,0}(R_h u(t)-P_h u(t)), \nonumber \\
L_h u(0)-U_h(0)&=&0. \nonumber
\end{eqnarray}
Then, because of (\ref{epsi}), $L_h u(t)- U_h(t)=\int_0^t e^{(t-s) A_{h,0}}O(\varepsilon_h) ds= O(\varepsilon_h).$
\end{proof}
We also have this finer result, which implies global order $s+1$ under more restrictive hypotheses.
\begin{theorem}
\label{teorcavbcsp}
Let us assume that $g(t)=0$ in (\ref{laibvp}), $u$ belongs to $C([0,T],Z)$, $f$ to $C^{s+2}([0,T], X)$, the interpolatory quadrature rule which is based on $\{c_i\}_{i=1}^s$ integrates exactly polynomials of degree less than or equal to $s$ and this bound holds
\begin{eqnarray}
\|k A_{h,0} \sum_{r=1}^{n-1} e^{r k A_{h,0}}\|_h
\le C , \quad 0 \le n k \le T.
\label{hparab}
\end{eqnarray}
Then,
\begin{enumerate}
\item[(i)]
$A_{h,0}^{-1}\rho_{h,n}=O(k^{s+2})$,
\item[(ii)]
$e_{h,n}=O(k^{s+1})$,
\item[(iii)]
$L_h u(t_n)- U_h^n=O(k^{s+1}+\varepsilon_h)$.
\end{enumerate}
where the constants in Landau notation are independent of $k$ and $h$.
\end{theorem}
\begin{proof}
To prove (i), it suffices to consider the following formula for the interpolation error which is valid when $f\in C^{s+1}$:
$$f(t_n+k \theta)-I(f(t_n+k \theta))=k^s \bigg[ f^{(s)}(t_n)\prod_{i=1}^s(\theta-c_i)+O(k) \bigg].$$
Then, substituting in (\ref{loccd}) and multiplying by $A_{h,0}^{-1}$,
\begin{eqnarray}
A_{h,0}^{-1}\rho_{h,n+1}&=&k^{s+2}\int_0^1 (1-\theta)(k(1-\theta) A_{h,0})^{-1} [e^{k(1-\theta)A_{h,0}}-I] L_h[f^{(s)}(t_n)\prod_{i=1}^s(\theta-c_i)+O(k)]d\theta
\nonumber \\
&&+k^{s+2} \int_0^1 k^{-1} A_{h,0}^{-1} L_h[f^{(s)}(t_n)\prod_{i=1}^s(\theta-c_i)+O(k)]d\theta \nonumber \\
&=& k^{s+2} \int_0^1 (1-\theta) \varphi_1(k (1-\theta) A_{h,0})L_h f^{(s)}(t_n)\prod_{i=1}^s(\theta-c_i)d\theta \nonumber \\
&&+k^{s+1} \bigg{(} \int_0^1 \prod_{i=1}^s(\theta-c_i) d\theta \bigg{)} A_{h,0}^{-1} L_h f^{(s)}(t_n) +O(k^{s+2})=O(k^{s+2}), \nonumber
\end{eqnarray}
where we have used (\ref{recurf}), (H1)-(H2) in \cite{acr2} and the fact that the integral in brackets vanishes because the interpolatory quadrature rule is exact for polynomials of degree $s$.

As for (ii), a summation-by-parts argument like that given in \cite{acr4} for splitting exponential methods also applies here because of hypothesis (\ref{hparab}) and the fact that $f\in C^{s+2}$. Finally, (iii) follows in the same way as in the proof of Theorem \ref{teorcavbc}.
\end{proof}

\begin{remark}
\label{norma2p}
Notice that, when $A_{h,0}$ is symmetric with negative eigenvalues and $\|\cdot\|_h$ is the discrete $L^2$-norm, $\|k A_{h,0} \sum_{r=1}^{n-1} e^{r k A_{h,0}}\|_h$ coincides with its spectral radius. As, for each eigenvalue $\lambda_h$,
$$
\bigg| k \lambda_h \sum_{r=1}^{n-1} e^{rk \lambda_h} \bigg|=k |\lambda_h|\frac{e^{k \lambda_h}-e^{t_n \lambda_h}}{1-e^{k \lambda_h}},
$$
and this is bounded in the negative real axis, (\ref{hparab}) follows. In fact, this bound has been proved in \cite{HO1} for analytic semigroups covering the case in which (\ref{laibvp}) corresponds to parabolic problems. Therefore it seems natural that it is also satisfied by a suitable space discretization.

\end{remark}

On the other hand, when $g \not\equiv 0$ in (\ref{laibvp}), the semidiscretized problem which arises is
\begin{eqnarray}
\left.
\begin{array}{lcl}
U'_h(t)&=&A_{h,0}U_h(t) + A_h Q_hg(t) +L_h Q_h(\partial f(t)-g'(t))+ P_hf(t),\\
U_h(0)&=&P_h u(0).
\end{array}
\right.
\nonumber
\end{eqnarray}
In a similar way as before, the local error would be given by
\begin{eqnarray}
\lefteqn{k \int_0^1 e^{k(1-\theta)A_{h,0}}\bigg[ A_h Q_h [g(t_n+k \theta)-I(g(t_n+k \theta))]} \nonumber \\
&& \hspace{2.5cm}+L_h Q_h [\partial f(t_n+k \theta)-g'(t_n+k \theta)-I(\partial f(t_n+k \theta)-g'(t_n+k \theta)) \nonumber \\
&&\hspace{2.5cm} + P_h [f(t_n+k \theta)-I(f(t_n+k \theta))]\bigg] d \theta. \label{lteca}
\end{eqnarray}
Again, when $g\in C^{s+1}([0,T],Y)$ and $f\in C^s([0,T], X)$, the error of interpolation will be $O(k^s)$. However, although $L_h Q_h$ and $P_h$ are bounded \cite{acr2}, $A_h Q_h$ is not bounded any more. That is why we state the following result which bounds in fact $A_{h,0}^{-1} \rho_{h,n}$ by using (HS) and which proof for the global error is the same as in Theorem \ref{teorcavbcsp}.
\begin{theorem}
\label{teorcanvbcsp}
Let us assume that $g(t)\not\equiv 0$ in (\ref{laibvp}), $u$ belongs to $C([0,T],Z)$, $g$ to $C^{s+2}([0,T],Y)$, $f$ to $C^{s+1}([0,T], X)$, and the bound (\ref{hparab}) holds.
Then,
\begin{enumerate}
\item[(i)]
$A_{h,0}^{-1}\rho_{h,n}=O(k^{s+1})$,
\item[(ii)]
$e_{h,n}=O(k^s)$,
\item[(iii)]
$L_h u(t_n)- U_h^n=O(k^s+ \varepsilon_h)$.
\end{enumerate}
where the constants in Landau notation are independent of $k$ and $h$.
\end{theorem}
\begin{remark}
\label{norma2}
As in Remark \ref{norma2p}, if $A_{h,0}$ is a symmetric matrix with negative eigenvalues and the discrete $L^2$-norm is considered, $\|k A_{h,0} e^{k(1-\theta)A_{h,0}}\|_h$ coincides with its spectral radius. As for each eigenvalue $\lambda_h$ of $A_{h,0}$,
$$
\int_0^1 k |\lambda_h| e^{k(1-\theta) \lambda_h} d \theta=\int_0^1 \frac{d}{d\theta}(e^{k(1-\theta)\lambda_h}) d \theta=1-e^{k \lambda_h}\le 1,
$$
considering this in the first part of (\ref{lteca}) explains that the local error $\rho_{h,n}$ behaves as $O(k^s)$ under the rest of hypotheses of Theorem \ref{teorcanvbcsp}.
\end{remark}
In any case, we want to remark in this section that accuracy has been lost with respect to the vanishing boundary conditions case since order reduction turns up at least for the local error and, in many cases, also for the global error.

\section{Suggested approach: Discretizing firstly in time and then in space}

In this section, we directly tackle the nonvanishing boundary conditions case by discretizing in a suitable way firstly in time and then in space. We will see that we manage to get at least the same order as with the classical approach when vanishing boundary conditions are present, but even  a much higher order some times.

 Let us suggest how to apply the exponential quadrature rule (\ref{qr}) directly to (\ref{laibvp}). When $g=0$, $B$ in (\ref{qr}) is directly substituted by $A_0$ and there is no problem because $e^{k A_0}$ and $\varphi_j(k A_0)$ have perfect sense over $X$. However, it has no sense to do that
when $g\neq 0$ because $A$ is not $A_0$ any more. For Lawson methods, for which just exponential functions appear, instead of $e^{\tau A_0} \alpha$, it was suggested in \cite{acr2} to consider $v_0(\tau)$ as the solution of
\begin{eqnarray}
v_0'(\tau)&=& A v_0(\tau), \nonumber \\
v_0(0)&=& \alpha, \nonumber \\
\partial v_0(\tau)&=&\sum_{l=0}^p \frac{\tau^l}{l!} \partial A^l \alpha, \label{ecv0}
\end{eqnarray}
whenever $\alpha\in D(A^p)$. In such a way, if $\alpha\in D(A^{p+1})$,
\begin{eqnarray}
v_0(\tau)=\sum_{l=0}^p \frac{\tau^l}{l!} A^l \alpha +\tau^{p+1} \varphi_{p+1}(\tau A_0) A^{p+1}\alpha,
\label{v0}
\end{eqnarray}
which resembles the formal analytic expansion of the exponential of $\tau A$ applied over $\alpha$.
In this manuscript then, whenever $\alpha\in D(A^p)$,  for $j=1,\dots,s$, instead of $\varphi_j(\tau A_0) \alpha$, we suggest to consider  the following functions :
\begin{eqnarray}
v_j(\tau)=\sum_{l=0}^{p-1} \frac{\tau^l}{(l+j)!} A^l \alpha +\tau^p \varphi_{p+j}(\tau A_0) A^p\alpha. \label{vjalpha}
\end{eqnarray}
This resembles the formal analytic expansion of $\varphi_j$ when evaluated at $\tau A$  and applied over  $\alpha$. (Notice that, for $j=0$, this would correspond to (\ref{v0}) changing $p$ by $p-1$. As the functions $\varphi_j$ are multiplied by $k$ in (\ref{qr}), we need one less term in this expansion.)

Therefore, imitating (\ref{qr}), we suggest to consider as continuous numerical approximation $u_{n+1}$ from the previous $u_n$,
\begin{eqnarray}
u_{n+1}=\lefteqn{\sum_{l=0}^p \frac{k^l}{l!} A^l u_n +k^{p+1} \varphi_{p+1}(k A_0) A^{p+1}u_n} \nonumber \\
&&+ k \sum_{i,j=1}^s a_{i,j} \bigg[ \sum_{l=0}^{p-1} \frac{k^l}{(l+j)!} A^l f(t_n+ c_i k) +k^p \varphi_{p+j}(k A_0) A^p f(t_n+ c_i k) \bigg].\label{qrcont}
\end{eqnarray}

It is not practical to discretize directly this formula in space since numerical differentiation is a badly-posed problem. However,  if we
seek a differential equation which the functions (\ref{vjalpha}) satisfy and apply a numerical integrator to it, there should be no problems. For that, let us first consider the following lemma.
\begin{lemma} For $j\ge 1$,
$$\frac{d}{d\tau} \varphi_j(\tau A_0)=(A_0-\frac{j}{\tau}I)\varphi_j(\tau A_0)+\frac{1}{(j-1)!\tau} I, \quad \tau>0.$$
\label{lemphi0}
\end{lemma}
\begin{proof} It suffices to differentiate considering (\ref{varphi}),
\begin{eqnarray}
\lefteqn{\frac{d}{d\tau}\bigg[ \frac{1}{\tau^j} \int_0^\tau e^{(\tau-\theta)A_0} \frac{\theta^{j-1}}{(j-1)!} d\theta\bigg]} \nonumber \\
&=&-\frac{j}{\tau^{j+1}}\int_0^\tau e^{(\tau-\theta)A_0} \frac{\theta^{j-1}}{(j-1)!}d \theta+\frac{1}{\tau^j} \frac{\tau^{j-1}}{(j-1)!}I +\frac{1}{\tau^j} \int_0^\tau A_0 e^{(\tau-\theta)A_0}\frac{\theta^{j-1}}{(j-1)!}d \theta \nonumber \\
&=&(A_0-\frac{j}{\tau}I)\varphi_j(\tau A_0)+\frac{1}{(j-1)!\tau} I. \nonumber
\end{eqnarray}
\end{proof}
From here, the next result follows:
\begin{lemma}
\label{lemvj}
The function $v_j(\tau)$ in (\ref{vjalpha}) satisfies the following initial boundary value problem:
\begin{eqnarray}
v_j'(\tau)&=& (A-\frac{j}{\tau}) v_j(\tau)+\frac{1}{(j-1)!\tau}\alpha, \quad \tau>0, \nonumber \\
v_j(0)&=& \frac{1}{j!}\alpha, \nonumber \\
\partial v_j(\tau)&=&\sum_{l=0}^{p-1} \frac{\tau^l}{(l+j)!} \partial A^l \alpha. \label{ecvj}
\end{eqnarray}
\end{lemma}
\begin{proof}
Notice that, using Lemma \ref{lemphi0},
\begin{eqnarray}
v_j'(\tau)&=&\sum_{l=1}^{p-1} \frac{l \tau^{l-1}}{(l+j)!}A^l \alpha+p \tau^{p-1} \varphi_{p+j}(\tau A_0) A^p \alpha \nonumber \\
&&+\tau^p[(A_0-\frac{p+j}{\tau}I)\varphi_{p+j}(\tau A_0)+\frac{1}{(p+j-1)!\tau}I] A^p \alpha \nonumber \\
&=& \sum_{l=1}^{p-1}\frac{l \tau^{l-1}}{(l+j)!}A^l \alpha+\frac{\tau^{p-1}}{(p+j-1)!} A^p \alpha-j \tau^{p-1} \varphi_{p+j}(\tau A_0) A^p \alpha+\tau^p A_0 \varphi_{p+j}(\tau A_0) A^p \alpha.
\nonumber
\end{eqnarray}
On the other hand,
\begin{eqnarray}
\lefteqn{(A-\frac{j}{\tau})v_j(\tau)+\frac{1}{(j-1)!\tau}\alpha} \nonumber \\
&=&\sum_{l=0}^{p-1}\frac{\tau^l}{(l+j)!}A^{l+1}\alpha+\tau^p A \varphi_{p+j}(\tau A_0) A^p \alpha
-j \sum_{l=1}^{p-1} \frac{\tau^{l-1}}{(l+j)!}A^l \alpha-j \tau^{p-1}\varphi_{p+j}(\tau A_0) A^p \alpha \nonumber \\
&=&\sum_{m=1}^{p-1} \frac{m}{(m+j)!}\tau^{m-1}A^m \alpha+\frac{\tau^{p-1}}{(p-1+j)!}A^p \alpha-j \tau^{p-1} \varphi_{p+j}(\tau A_0) A^p \alpha+\tau^p A \varphi_{p+j}(\tau A_0) A^p \alpha, \nonumber
\end{eqnarray}
where the change $m=l+1$ has been used in the second line for the first sum. Therefore, the lemma is proved taking also into account that $\varphi_{p+j}(\tau A_0) A^p \alpha \in D(A_0)$.
\end{proof}

With Lawson methods \cite{acr2}, starting from a previous approximation $U_{h,n}$ to $P_h u(t_n)$ and discretizing (\ref{ecv0}) in space with $\alpha=u(t_n)$ in $\partial v_0(\tau)$ led to a term like
\begin{eqnarray}
V_{h,n,0}(k)&=&e^{k A_{h,0}} U_{h,n}+\sum_{j=1}^p k^{j} \varphi_j(k A_{h,0})[A_h Q_h \partial A^{j-1} u(t_n)-L_h Q_h \partial A^j u(t_n)]
\nonumber \\&&
+k^{p+1} \varphi_{p+1}(k A_{h,0}) A_h Q_h \partial A^p u(t_n), \label{Vh0}
\end{eqnarray}
which is the approximation which corresponds to the first term in (\ref{qr}).

For the rest of the terms in (\ref{qr}), we suggest to discretize (\ref{ecvj}) in space with $\alpha=f(t_n+c_i k)$. In such a way, the following system turns up:
\begin{eqnarray}
V_{h,j,n,i}'(\tau)+L_h Q_h \partial \hat{v}_{j,n,i}'(\tau)&=& A_{h,0} V_{h,j,n,i}(\tau)+A_h Q_h \partial \hat{v}_{j,n,i}(\tau)-\frac{j}{\tau}[V_{h,j,n,i}(\tau)+L_h Q_h \partial \hat{v}_{j,n,i}(\tau)] \nonumber \\
&&+\frac{1}{(j-1)!\tau}[P_h f(t_n+c_i k)+ L_h Q_h \partial f(t_n+c_i k)], \nonumber \\
V_{h,j,n,i}(0)+L_h Q_h \partial \hat{v}_{j,n,i}(0)&=&\frac{1}{j!} L_h f(t_n+c_i k),
\nonumber
\end{eqnarray}
where
\begin{eqnarray}
\hat{v}_{j,n,i}(\tau)=\sum_{l=0}^{p-1} \frac{\tau^l}{(l+j)!} A^l f(t_n+c_i k).
\nonumber
\end{eqnarray}
This can be rewritten as
\begin{eqnarray}
\hspace{0.2cm}V_{h,j,n,i}'(\tau)&=& (A_{h,0}-\frac{j}{\tau} I) V_{h,j,n,i}(\tau)+A_h Q_h \partial \hat{v}_{j,n,i}(\tau)+\frac{1}{(j-1)!\tau}P_h f(t_n+c_i k) \nonumber \\
&&+L_h Q_h \partial [\frac{1}{(j-1)!\tau} f(t_n+c_i k)-\frac{j}{\tau} \hat{v}_{j,n,i}(\tau)- \hat{v}_{j,n,i}'(\tau)], \nonumber \\
\hspace{0.2cm}V_{h,j,n,i}(0)&=&\frac{1}{j!} P_h f(t_n+c_i k).
\label{ecVhj}
\end{eqnarray}
With the same arguments as in Lemma \ref{lemphi0}, $\varphi_j(\tau A_{h,0})P_h f(t_n+c_i k)$ is the solution of
\begin{eqnarray}
W_{h,j,n,i}'(\tau)&=& (A_{h,0}-\frac{j}{\tau} I) W_{h,j}(\tau)+\frac{1}{(j-1)!\tau}P_h f(t_n+c_i k), \nonumber \\
W_{h,j,n,i}(0)&=&\frac{1}{j!} P_h f(t_n+c_i k).
\nonumber
\end{eqnarray}
Therefore, in order to solve (\ref{ecVhj}), we are interested in finding
\begin{eqnarray}
Z_{h,j,n,i}(\tau)=V_{h,j,n,i}(\tau)-W_{h,j,n,i}(\tau),
\label{zhj}
\end{eqnarray}
which is the solution of
\begin{eqnarray}
Z_{h,j,n,i}'(\tau)&=& (A_{h,0}-\frac{j}{\tau} I) Z_{h,j,n,i}(\tau)+A_h Q_h \partial \hat{v}_{j,n,i}(\tau)
\nonumber \\
&&+L_h Q_h \partial[\frac{1}{(j-1)!\tau} f(t_n+c_i k)-\frac{j}{\tau} \hat{v}_{j,n,i}(\tau)- \hat{v}_{j,n,i}'(\tau)], \nonumber \\
Z_{h,j,n,i}(0)&=&0.
\label{Zhjni}
\end{eqnarray}
Now,  using the first line of (\ref{ecvj}) for the boundary with $\alpha=f(t_n+c_ik)$,
\begin{eqnarray}
\lefteqn{\partial [\frac{1}{(j-1)!\tau} f(t_n+c_i k)-\frac{j}{\tau} \hat{v}_{j,n,i}(\tau)- \hat{v}_{j,n,i}'(\tau)]=-\partial A v_{j,n,i}(\tau) } \nonumber \\
&=&-\sum_{l=0}^{p-1} \frac{\tau^l}{(l+j)!} \partial A^{l+1} f(t_n+c_i k)-\tau^p \partial A_0 \varphi_{p+j}(\tau A_0)A^p f(t_n+c_ik), \nonumber
\end{eqnarray}
the fact that
\begin{eqnarray}
\lefteqn{\tau^p A_0 \frac{1}{\tau^{p+j}}\int_0^{\tau} e^{(\tau-\sigma)A_0} \frac{\sigma^{p+j-1}}{(p+j-1)!}A^p f(t_n+c_i k) d \sigma} \nonumber \\
&&=-\frac{1}{\tau^j}e^{(\tau-\sigma)A_0} \frac{\sigma^{p+j-1}}{(p+j-1)!}|_{\sigma=0}^{\sigma=\tau}A^p f(t_n+c_i k)+\frac{1}{\tau^j} \int_0^{\tau} e^{(\tau-\sigma)A_0}
\frac{\sigma^{p+j-2}}{(p+j-2)!} A^p f(t_n+c_i k) d \sigma \nonumber \\
&&=-\frac{\tau^{p-1}}{(p+j-1)!} A^p f(t_n+c_i k)+\tau^{p-1} \varphi_{p+j-1}(\tau A_0) A^p f(t_n+c_i k), \nonumber
\end{eqnarray}
and that the boundary of the second term vanishes, it follows that
\begin{eqnarray}
\partial[\frac{1}{(j-1)!\tau} f(t_n+c_i k)-\frac{j}{\tau} \hat{v}_{j,n,i}(\tau)- \hat{v}_{j,n,i}'(\tau)]=
-\sum_{l=0}^{p-2} \frac{\tau^l}{(l+j)!}\partial A^{l+1} f(t_n+c_i k).
\nonumber
\end{eqnarray}
Using this in (\ref{Zhjni}),
\begin{eqnarray}
\lefteqn{Z_{h,j,n,i}(\tau)} \nonumber \\
&=&\int_0^\tau e^{\int_{\theta}^\tau (A_{h,0}-\frac{j}{\sigma}I)d\sigma}\sum_{l=0}^{p-2} \frac{\theta^l}{(l+j)!}\bigg[A_h Q_h \partial  A^l f(t_n+c_i k)
-L_h Q_h \partial A^{l+1} f(t_n+c_i k)
\nonumber \\
&&\hspace{6cm}+\frac{\theta^{p-1}}{(p-1+j)!}A_h Q_h \partial A^{p-1} f(t_n+c_i k)
\bigg]d\theta
\nonumber \\
&=&\sum_{l=0}^{p-2} \int_0^\tau e^{A_{h,0}(\tau-\theta)}\frac{\theta^{j+l}}{\tau^j(l+j)!}[A_h Q_h \partial  A^l f(t_n+c_i k)
-L_h Q_h \partial A^{l+1} f(t_n+c_i k)]d\theta \nonumber \\
&&+\int_0^\tau e^{A_{h,0}(\tau-\theta)}\frac{\theta^{p-1+j}}{\tau^j(p-1+j)!}A_h Q_h \partial  A^{p-1} f(t_n+c_i k) d\theta \nonumber \\
&=&\sum_{l=0}^{p-2} \tau^{l+1} \varphi_{j+l+1}(\tau A_{h,0})[A_h Q_h \partial  A^l f(t_n+c_i k)-L_h Q_h \partial A^{l+1} f(t_n+c_i k)]
\nonumber \\
&&+\tau^p \varphi_{p+j}(\tau A_{h,0}) A_h Q_h \partial A^{p-1}f(t_n+c_i k).
\end{eqnarray}
Therefore, using (\ref{zhj}),
\begin{eqnarray}
\lefteqn{\hspace{-1cm}V_{h,j,n,i}(k)=\varphi_j(k A_{h,0}) P_h f(t_n+c_i k)} \nonumber \\
&&+\sum_{l=0}^{p-2} k^{l+1} \varphi_{j+l+1}(k A_{h,0})[A_h Q_h \partial  A^l f(t_n+c_i k)
-L_h Q_h \partial A^{l+1} f(t_n+c_i k)],
\nonumber \\
&&+k^p \varphi_{p+j}(k A_{h,0})A_h Q_h \partial A^{p-1} f(t_n+c_i k),
 \label{Vhij}
\end{eqnarray}
and the overall exponential quadrature rule would be given by
\begin{eqnarray}
U_{h,0}&=&P_h u_0, \nonumber \\
U_{h,n+1}&=&V_{h,n,0}(k)+k\sum_{i,j=1}^s a_{i,j} V_{h,j,n,i}(k),
\label{ff}
\end{eqnarray}
with $V_{h,n,0}(k)$ in (\ref{Vh0}) and $V_{h,j,n,i}(k)$ in (\ref{Vhij}).

\subsection{Time semidiscretization error}
\label{tde}
Let us first study just the error after time discretization.
The local truncation error  is well-known to be given by $\rho_n=u(t_{n+1})-\bar{u}_{n+1}$, where $\bar{u}_{n+1}$ is given by expression (\ref{qrcont}) substituting  $u_n$ by $u(t_n)$.

Let us first consider the following general result, which will allow to conclude more particular results depending on the choice of the values $\{c_i\}_{i=1}^s$.
\begin{lemma}
Under the assumptions of regularity (\ref{reg}), the local truncation error  satisfies
$$
\rho_n=\sum_{m=1}^p k^m \bigg[ \sum_{r=0}^{m-1} \big( \frac{1}{m!}-\frac{1}{r!}\sum_{l=1}^s \frac{1}{(m-r-1+l)!}\sum_{i=1}^s c_i^r a_{il} \big) A^{m-r-1}f^{(r)}(t_n) \bigg]+O(k^{p+1}).
$$
\label{lemmarhon}
\end{lemma}
\begin{proof}
Notice that $\bar{u}_{n+1}$ can be written as
\begin{eqnarray}
\bar{u}_{n+1}&=&u(t_n)+\sum_{j=1}^p k^j \bigg[ \frac{1}{j!} A^j u(t_n)+\sum_{i,l=1}^s a_{i,l} \frac{1}{(j-1+l)!}A^{j-1} f(t_n+c_i k) \bigg]+O(k^{p+1}) \nonumber \\
&=&
u(t_n)+\sum_{j=1}^p k^j \bigg[ \frac{1}{j!} A^j u(t_n)+\sum_{i,l=1}^s a_{i,l} \frac{1}{(j-1+l)!} \sum_{r=0}^{p-j} \frac{c_i^r k^r}{r!}A^{j-1} f^{(r)}(t_n) \bigg]+O(k^{p+1}) \nonumber \\
&=&
u(t_n)+\sum_{m=1}^p k^m \bigg[ \frac{1}{m!} A^m u(t_n)+\sum_{i,l=1}^s a_{i,l} \sum_{r=0}^{m-1} \frac{c_i^r}{(m-r-1+l)!r!} A^{m-r-1} f^{(r)}(t_n) \bigg]+O(k^{p+1}),
\nonumber
\end{eqnarray}
where the Taylor expansion of $f(t_n+c_i k)$ has been used as well as changes of subindexes.

As, according to (\ref{ujt}),
\begin{eqnarray}
u(t_{n+1})&=&u(t_n)+\sum_{m=1}^p \frac{k^m}{m!} u^{(m)}(t_n) \nonumber \\
&=&u(t_n)+\sum_{m=1}^p \frac{k^m}{m!} [A^m u(t_n)+\sum_{r=0}^{m-1} A^{m-r-1}f^{(r)}(t_n)]+O(k^{p+1}),
\nonumber
\end{eqnarray}
the result follows.
\end{proof}

\begin{theorem}
If $p=s$ in  (\ref{reg}) and (\ref{qrcont}), for any nodes $\{ c_i \}_{i=1}^s$, $\rho_n=O(k^{s+1})$.
\label{locps}
\end{theorem}
\begin{proof}
It suffices to take into account that any polynomial of degree $\le s-1$ coincides with its interpolant on the nodes $\{ c_i\}_{i=1}^s$. Therefore,  for $r\le s-1$, $\sum_{i=1}^s c_i^r l_i (\theta)=\theta^r$. Using (\ref{litheta}), this implies that
$$
\sum_{i=1}^s c_i^r \frac{a_{i,l}}{(l-1)!}=\bigg{\{} \begin{array}{lcl} 0 &\mbox{if} & l \neq r+1 \\ 1 &\mbox{if} & l=r+1, \end{array}
$$
or equivalently, for $r\le s-1$,
$$
\sum_{i=1}^s c_i^r a_{i,r+1}=r!, \quad \sum_{i=1}^s c_i^r a_{i,l}=0, \mbox{ whenever }l\neq r+1.
$$
Substituting this in the expression for $\rho_n$ in Lemma \ref{lemmarhon} with $p=s$, all the terms in brackets vanish and the result follows.
\end{proof}

\begin{theorem}
If $p=s+1$ in  (\ref{reg}) and (\ref{qrcont}) and the nodes $\{ c_i \}_{i=1}^s$ are such that the interpolatory quadrature rule which is based on them is exact for polynomials of degree $\le s$, $\rho_n=O(k^{s+2})$.
\label{locps1}
\end{theorem}
\begin{proof} With the same argument as in the previous lemma, all the terms in brackets in the expression of $\rho_n$ in Lemma \ref{lemmarhon} vanish for $m\le s$. Then, for $m=s+1$, the term in parenthesis vanishes for the same reason when $r \le s-1$. It just suffices to see what happens when $m=s+1$ and $r=s$. But, as the quadrature rule which is based on $\{c_i\}_{i=1}^s$ is assumed to be exact for the polynomial $\theta^s$,
$$\frac{1}{s+1}=\int_{0}^1 \theta^s d\theta=\int_{0}^1\sum_{i=1}^s c_i^s l_i(\theta) d \theta=\sum_{i=1}^s c_i^s \sum_{l=1}^s \frac{a_{i,l}}{l!}.$$
From this, the result also directly follows.
\end{proof}
We now state the following much more general result:
\begin{theorem}
Whenever the nodes $\{ c_i \}_{i=1}^s$ are such that the interpolatory quadrature rule which is based on them is exact for polynomials of degree $\le p-1$, considering that value of $p$ in (\ref{reg}) and (\ref{qrcont}), $\rho_n=O(k^{p+1})$.
\label{locp}
\end{theorem}
\begin{proof}
It suffices to notice that, for $0\le r\le m-1$, with $m\le p$, due to the hypothesis,
\begin{eqnarray}
\lefteqn{\hspace{-2cm}\int_0^1 \int_0^{u_1} \dots \int_0^{u_{m-r-1}} \theta^r d\theta d u_{m-r-1}\dots du_1}
\nonumber \\
\hspace{2cm}&=&\int_0^1 \int_0^{u_1} \dots \int_0^{u_{m-r-1}} \sum_{i=1}^s c_i^r l_i(\theta) d\theta d u_{m-r-1}\dots du_1.
\nonumber
\end{eqnarray}
Now, the left-hand side term above can inductively be proved to be
$$
\frac{1}{r+1} \frac{1}{r+2} \dots \frac{1}{m},
$$
and the right-hand side can be written as
\begin{eqnarray}
\lefteqn{\int_0^1 \int_0^{u_1} \dots \int_0^{u_{m-r-1}} \sum_{i=1}^s c_i^r \sum_{l=1}^s a_{i,l} \frac{\theta^{l-1}}{(l-1)!} d\theta d u_{m-r-1} \dots du_1} \nonumber \\
&=&\sum_{l=1}^s \big( \sum_{i=1}^s c_i^r a_{i,l}\big) \int_0^1 \int_0^{u_1} \dots \int_0^{u_{m-r-1}} \frac{\theta^{l-1}}{(l-1)!} d\theta d u_{m-r-1} \dots du_1 \nonumber \\
&=&\sum_{l=1}^s \big( \sum_{i=1}^s c_i^r a_{i,l}\big) \frac{1}{(l-1+m-r)!}.
\nonumber
\end{eqnarray}
Then, using Lemma \ref{lemmarhon}, the result directly follows.
\end{proof}
From this, the following interesting results are achieved:
\begin{corollary}
\label{corol}
\begin{list}{}{}
\item[(i)] For the $s$ nodes corresponding to a Gaussian quadrature rule, considering $p=2s$ in (\ref{reg}) and (\ref{qrcont}), $\rho_n=O(k^{2s+1})$.
\item[(ii)]For the $s$ nodes corresponding to a Gaussian-Lobatto quadrature rule, considering $p=2s-2$ in (\ref{reg}) and (\ref{qrcont}), $\rho_n=O(k^{2s-1})$.
\end{list}
\end{corollary}
\begin{remark}
\label{gl}
Due to the fact that the last node of one step is the first of the following, the nodes corresponding to the Gaussian-Lobatto quadrature rule have the advantage that just $s(s-1)$ (instead of $s^2$)  terms of the form $V_{h,n,j,i}$ must be calculated in (\ref{ff}).

\end{remark}

\subsection{Full discretization error}
Let us also consider the error which arises when discretizing (\ref{ecv0}) and (\ref{ecvj}) in space.
\subsubsection{Local error}
To define the local error after full discretization, we consider
$$
\bar{U}_{h,n+1}=\bar{V}_{h,n,0}(k)+k \sum_{i,j=1}^s a_{i,j} \bar{V}_{h,n,j,i}(k),
$$
where
\begin{enumerate}
\item[(i)]
$\bar{V}_{h,n,0}(\tau)$ is the solution of
\begin{eqnarray}
\bar{V}_{h,n,0}'(\tau)+L_h Q_h \partial \hat{v}_{n,0}'(\tau)&=&A_{h,0} \bar{V}_{h,n,0}(\tau)+ A_h Q_h \partial \hat{v}_{n,0}(\tau), \nonumber \\
\bar{V}_{h,n,0}(0)&=&R_h u(t_n), \nonumber
\end{eqnarray}
with $\hat{v}_{n,0}(\tau)=\sum_{l=0}^p \frac{\tau^l}{l!} A^l u(t_n)$.
\item[(ii)]
$\bar{V}_{h,n,j,i}(\tau)$ is the solution of (\ref{ecVhj}) substituting $P_h f(t_n+c_i k)$ by $R_h f(t_n+c_i k)$. More precisely,
\begin{eqnarray}
\bar{V}_{h,j,n,i}'(\tau)&=& (A_{h,0}-\frac{j}{\tau} I) \bar{V}_{h,j,n,i}(\tau)+A_h Q_h \partial \hat{v}_{j,n,i}(\tau)
\nonumber \\
&&+\frac{1}{(j-1)!\tau}R_h f(t_n+c_i k) \nonumber \\
&&+L_h Q_h \partial[\frac{1}{(j-1)!\tau} f(t_n+c_i k)-\frac{j}{\tau} \hat{v}_{j,n,i}(\tau)- \hat{v}_{j,n,i}'(\tau)], \nonumber \\
\bar{V}_{h,j,n,i}(0)&=&\frac{1}{j!} R_h f(t_n+c_i k).
\label{ecbVhj}
\end{eqnarray}
\end{enumerate}
Then, we define $\rho_{h,n}=R_h u(t_{n+1})-\bar{U}_{h,n+1}$ and the following is satisfied.
\begin{theorem}
\label{theolocalerrorfull}
Let us assume that, apart from hypotheses of Section \ref{prel}, $u$ and $f$ in (\ref{laibvp}) satisfy
\begin{eqnarray}
 A^j u  \in C([0,T],Z), \quad  j=0,\dots,p+1,  \qquad  A^j f \in C([0,T], Z), \quad  j=
0,\dots,p.
\label{regs}
\end{eqnarray}
Then, $\rho_{h,n}=O(k \varepsilon_h+\|\rho_n\|),$ where the constant in Landau notation is independent of $k$ and $h$ and the bounds in Section \ref{tde} hold for $\rho_n$.
\end{theorem}
\begin{proof}
Because of definition,
\begin{eqnarray}
\rho_{h,n}&=&(R_h u(t_{n+1})-R_h \bar{u}_{n+1})+(R_h \bar{u}_{n+1}-\bar{U}_{h,n+1})
\nonumber \\
&=&R_h \rho_n+(R_h \bar{u}_{n+1}-\bar{U}_{h,n+1}),
\label{pr1}
\end{eqnarray}
where $\bar{u}_n$ and $\rho_n$ are those defined in Section \ref{tde}. The fact that (\ref{regs}) is satisfied implies that $\bar{u}_{n+1}$ belongs to $Z$ and therefore $\rho_n \in Z$. Moreover, $\|\rho_n\|_Z=O(\|\rho_n\|)$ and, using the same proof as that of Theorem 11 in \cite{acr2},
\begin{eqnarray}
R_h \rho_n=O(\|\rho_n\|).
\label{pr2}
\end{eqnarray}
On the other hand,
\begin{eqnarray}
R_h \bar{u}_{n+1}-\bar{U}_{h,n+1}=R_h \bar{v}_{0,n}-\bar{V}_{h,n,0}(k)+k\sum_{i,j=1}^s a_{i,j}[ R_h v_{j,n,i}(k)-\bar{V}_{h,j,n,i}(k)],
\label{pr3}
\end{eqnarray}
where $\bar{v}_{0,n}$ corresponds to (\ref{v0}) with $\alpha=u(t_n)$  and $v_{j,n,i}(\tau)$ corresponds to (\ref{vjalpha}) with $\alpha=f(t_n+c_i k)$.
In the same way as in the proof of Theorem 4.4 in \cite{acr2},
\begin{eqnarray}
R_h \bar{v}_{0,n}-\bar{V}_{h,n,0}(k)=O(k \varepsilon_h).
\label{pr4}
\end{eqnarray}
Moreover, using Lemma \ref{lemvj},
\begin{eqnarray}
R_h v_{j,n,i}'(\tau)&=& R_h A v_{j,n,i}(\tau)-\frac{j}{\tau}R_h v_{j,n,i}(\tau)+\frac{1}{(j-1)!\tau} R_h f(t_n+c_i k) \nonumber \\
&=& P_h A v_{j,n,i}(\tau)+(R_h-P_h) A v_{j,n,i}(\tau)-\frac{j}{\tau} R_h v_{j,n,i}(\tau)+\frac{1}{(j-1)! \tau} R_h f(t_n+c_i k) \nonumber \\
&=& A_{h,0} R_h v_{j,n,i}(\tau)+A_h Q_h \partial \hat{v}_{j,n,i}(\tau)-L_h Q_h \partial A v_{j,n,i}(\tau)
\nonumber \\
&&+(R_h-P_h) A v_{j,n,i}(\tau)-\frac{j}{\tau} R_h v_{j,n,i}(\tau)+\frac{1}{(j-1)! \tau} R_h f(t_n+c_i k), \nonumber
\end{eqnarray}
and making the difference with (\ref{ecbVhj}), it follows that
\begin{eqnarray}
R_h v_{j,n,i}'(\tau)-\bar{V}_{h,n,j,i}'(\tau)=(A_{h,0}-\frac{j}{\tau}I)(R_h v_{j,n,i}(\tau)-V_{h,n,j,i}(\tau))+(R_h -P_h) A v_{j,n,i},
\nonumber
\end{eqnarray}
where we have used that
$$\partial[A v_{j,n,i}(\tau)+\frac{1}{(j-1)! \tau} f(t_n+c_i k)-\frac{j}{\tau} \partial \hat{v}_{j,n,i}(\tau)-\partial \hat{v}_{j,n,i}'(\tau)]=0$$
because of Lemma \ref{lemvj}. Now, due to the same lemma and (\ref{ecbVhj}), $R_h v_{j,n,i}(0)-\bar{V}_{h,n,j,i}(0)=0$, and therefore
\begin{eqnarray}
R_h v_{j,n,i}(k)-\bar{V}_{h,n,j,i}(k)&=&\int_0^k e^{(k-\tau)A_{h,0}}\frac{\tau^j}{k^j}(R_h-P_h) A v_{j,n,i}(\tau) d \tau
\nonumber \\
&=& k \varphi_{j+1}(k A_{h,0}) O( \varepsilon_h)= O(k \varepsilon_h). \label{pr5}
\end{eqnarray}
Here we have used that $A v_{j,n,i} \in Z$ because of (\ref{regs}) and Lemma 3.3 in \cite{acr2}. Finally, gathering (\ref{pr1})--(\ref{pr5}), the result follows.
\end{proof}

\subsubsection{Global error}
We now study the global error, which we define as $e_{h,n}=P_h u(t_n)-U_{h,n}$.

\begin{theorem}
\label{theoglobalerrorfull}
Under the same assumptions of Theorem \ref{theolocalerrorfull},
\begin{eqnarray*}
e_{h,n}=O( \frac{1}{k} \max_{0\le l \le n-1} \|\rho_l\|+\varepsilon_h),
\end{eqnarray*}
where the constant in Landau notation is independent of $k$ and $h$ and the bounds in Section \ref{tde} hold for $\rho_l$.
\end{theorem}

\begin{proof} As in the proof of Theorem 4.5 in \cite{acr2},
\begin{eqnarray}
e_{h,n+1}&=&(P_hu(t_{n+1})-R_h u(t_{n+1}))+R_hu(t_{n+1})-U_{h,n+1} \nonumber \\
&=&O(\varepsilon_h)+R_hu(t_{n+1})-U_{h,n+1}.
\label{fdecomp}
\end{eqnarray}
The difference is that now, using (\ref{ff}),
\begin{eqnarray*}
R_hu(t_{n+1})-U_{h,n+1}
&=&\rho_{h,n}+\overline{U}_{h,n+1}-U_{h,n+1}
\nonumber \\
&=&\rho_{h,n}+\bar{V}_{h,n,0}(k)-V_{h,n,0}(k)+k \sum_{i,j=1}^s a_{ij} (\bar{V}_{h,j,n,i}(k)-V_{h,j,n,i}(k)).
\end{eqnarray*}
As in \cite{acr2},
\begin{eqnarray*}
\bar{V}_{h,n,0}-V_{h,n,0}=e^{kA_{h,0}}(R_hu(t_n)-U_{h,n}).
\end{eqnarray*}
As for $\bar{V}_{h,j,n,i}(k)-V_{h,j,n,i}(k)$, making the difference between (\ref{ecbVhj}) and (\ref{ecVhj}),
\begin{eqnarray}
\bar{V}_{h,j,n,i}'(\tau)-V_{h,j,n,i}'(\tau)
&=&(A_{h,0}-\frac{j}{\tau})(\bar{V}_{h,j,n,i}(\tau)-V_{h,j,n,i}(\tau))+\frac{1}{(j-1)!\tau}(R_h-P_h) f(t_n+c_i k),
\nonumber \\
\bar{V}_{h,j,n,i}(0)-V_{h,j,n,i}(0)&=&\frac{1}{j!}(R_h-P_h) f(t_n+c_i k).
\nonumber
\end{eqnarray}
Considering then an analogue of Lemma \ref{lemphi0} substituting $A_0$ by $A_{h,0}$  and taking into account that $\varphi_j(0)=1/j!$ (\ref{recurf}),
\begin{eqnarray*}
\bar{V}_{h,j,n,i}(k)-V_{h,j,n,i}(k)=\varphi_j(kA_{h,0})(R_h-P_h)f(t_n+c_i k)=O(\varepsilon_h).
\end{eqnarray*}
Therefore,
\begin{eqnarray*}
R_hu(t_{n+1})-U_{h,n+1}= e^{k A_{h,0}}(R_h u(t_n)-U_{h,n})+\rho_{h,n}+O(k \varepsilon_h).
\end{eqnarray*}
This implies that
$$
R_h u(t_{n+1})-U_{h,n+1}=e^{t_{n+1} A_{h,0}}(R_h u(0)-U_{h,0})+O( \frac{1}{k} \max_{0\le l \le n} \|\rho_{h,l}\|+ \varepsilon_h),
$$
which, together with the first line of (\ref{ff}), (\ref{epsi}), (\ref{fdecomp}) and Theorem \ref{theolocalerrorfull}, implies the result.

\end{proof}

\section{Numerical experiments}
In this section we will show some numerical experiments which corroborate the previous results. For that, we have considered parabolic problems with homogeneous and non-homogeneous Dirichlet boundary conditions for which $X=L^2(\Omega)$ for a certain spatial domain $\Omega$ and $g\in H^{\frac{1}{2}}(\Omega)$. The fact that these problems can be well fitted under the theory of abstract IBVPs is well justified in \cite{acr2,quarteroniv}. Moreover, other types of boundary conditions can also be considered although we restrict here to Dirichlet boundary conditions just for the sake of brevity.

As for the space discretization, we have considered here both the standard symmetric 2nd-order finite differences and collocation spectral methods in $1$ dimension. For the former, it was already well justified in \cite{acr2} that the hypotheses which are required on the space discretization are satisfied, at least for the discrete $L^2$-norm, $Z=H^4(\Omega)$ and $\varepsilon_h=O(h^2)$. Besides, a discrete maximum principle (hypothesis (HS)) is well-known to apply \cite{S}. With the collocation spectral methods, those hypotheses are also valid with the discrete $L^2$-norm associated to the corresponding Gaussian-Lobatto quadrature rule ($\|\cdot\|_{h,GL}$), $Z=H^m(\Omega)$ and $\varepsilon_h=O(J^{2-m})$ \cite{acj,BM}, where $J+1$ is the number of collocation nodes, which is clearly inversely proportional to the diameter space grid $h$. In such a way, the more regular the functions are, the quicker the numerical solution of the elliptic problems converges to the exact solution.

Besides, although in the collocation case the matrix $A_{h,0}$ is not symmetric any more, Remarks \ref{norma2p} and \ref{norma2} still apply. Notice that, for every matrix $B$ of dimension $(J-1)\times(J-1)$,
\begin{eqnarray}
\|B\|_{h,GL}=\|D_J B D_J^{-1}\|_h
\label{BGL}
\end{eqnarray}
where $D_J$ denotes the diagonal matrix which contains the square root of the coefficients of the quadrature rule corresponding to the interior Gauss-Lobatto nodes $\{ x_j \}_{j=1}^{J-1}$. (We will denote them by $\{\alpha_j\}_{j=1}^{J-1}$.) Because of this, when $D_J B D_J^{-1}$ is symmetric, $\|B\|_{h,GL}=\rho(B)$. The fact that $D_J A_{h,0} D_J^{-1}$ is symmetric comes from the following: Notice that $(A_h)_{i,j}=L_j''(x_i)$ where $\{L_j(x)\}$ are the Lagrange polynomials associated to the interior Gauss-Lobatto nodes and those at the boundary. As $\{L_j(x)\}_{j=1}^{J-1}$ vanish at the boundary, integrating by parts, for every $i,j\in \{1,\dots,J-1\}$,
$$\int L_j''(x) L_i(x) dx=-\int L_j'(x) L_i'(x) dx.$$
As the integrand in the left-hand side is a polynomial of degree $2J-2$, the corresponding Gaussian-Lobatto quadrature rule integrates it exactly. Therefore,
$$
\alpha_i L_j''(x_i)=-\int L_j'(x) L_i'(x)dx= \alpha_j L_i''(x_j),$$
where, for the last equality, the role of $i$ and $j$ has been interchanged. From this, and using (\ref{BGL}) again,
$$\|k A_{h,0} \sum_{r=1}^{n-1} e^{r k A_{h,0}}\|_{h,GL}=\|k D_J A_{h,0} D_J^{-1} e^{k(1-\theta)D_J A_{h,0} D_J^{-1}}\|_h.$$
As $D_J A_{h,0} D_J^{-1}$ is symmetric, the matrix inside $\|\cdot\|_h$ is also symmetric and therefore
$$ \|k A_{h,0} \sum_{r=1}^{n-1} e^{r k A_{h,0}}\|_{h,GL}= \rho(k D_J A_{h,0} D_J^{-1} \sum_{r=1}^{n-1} e^{r k D_J A_{h,0} D_J^{-1}})= \rho (k A_{h,0} \sum_{r=1}^{n-1} e^{r kA_{h,0}}).$$

Secondly, the eigenvalues of $A_{h,0}$ are negative. This is due to the following: For every polynomial which vanishes at the boundary such that $p(x)\not\equiv 0$,
$$\int p''(x) p(x) dx=-\int [p'(x)]^2 dx<0.$$
Considering $p(x)=\sum_{i=1}^{J-1} \beta_i L_i(x)$ and using the Gauss-Lobatto quadrature rule and the definition of Lagrange polynomials,
$$
\sum_{k=1}^{J-1} \alpha_k (\sum_{i=1}^{J-1} \beta_i L_i''(x_k))(\sum_{j=1}^{J-1} \beta_j L_j(x_k))=\sum_{i,k=1}^{J-1} \alpha_k \beta_i \beta_k L_i''(x_k)<0.$$
This can be rewritten as $\vec{\beta}^T D_J^2 A_{h,0} \vec{\beta}<0$ for every vector $\vec{\beta}\neq \vec{0}$, or equivalently, $(D_J\vec{\beta})^T D_J  A_{h,0} D_J^{-1} (D_J \vec{\beta})<0$, which implies that $D_J A_{h,0} D_J^{-1}$ has negative eigenvalues and so has $A_{h,0}$.

For both types of discretizations which have been considered here, $L_h Q_h \partial \equiv 0$ and therefore formulas (\ref{Vh0}) and (\ref{Vhij}) simplify a little bit. However, other possible discretizations (as those considered in \cite{acr2}) are also possible, for which that simplification cannot be made.

In all cases, we have considered the one-dimensional problem
\begin{eqnarray}
u_t(x,t)&=&u_{xx}(x,t)+f(x,t), \quad  0\le t \le 1, \quad 0\le x \le 1,\nonumber \\
u(x,0)&=&u_0(x), \nonumber \\
u(0,t)&=& g_0(t), \quad u(1,t)= g_1(t), \label{pr}
\end{eqnarray}
with the corresponding functions $f$, $u_0$, $g_0$ and $g_1$ which make that $u(x,t)=x(1-x)e^{-t}$ or $u(x,t)=e^{x-t}$ are solutions of the problem.
These functions satisfy regularity hypotheses (\ref{reg}) and (\ref{regs}) for any natural number $p$.

\subsection{Trapezoidal rule}

We begin by considering the trapezoidal rule in time and the second-order finite differences in space. We have considered $h=10^{-3}$ so that the error in space is negligible. The trapezoidal rule corresponds to $s=2$ but is just exact for polynomials of degree $\le 1$. Therefore, one of the hypothesis of Theorem \ref{teorcavbcsp} is not satisfied and we can just apply Theorem \ref{teorcavbc} when discretizating firstly in space and then in time with the solution which satisfies $g_0(t)=g_1(t)=0$. That theorem states that, with respect to the time stepsize $k$, the local and global error should show orders $3$ and $2$ respectively and we can check that really happens in Table \ref{t1}. For the same problem, but applying the technique which is suggested in this paper (\ref{ff}) with $p=2$, Theorems \ref{locps}, \ref{theolocalerrorfull} and \ref{theoglobalerrorfull} state that also the local and global error should show orders $3$ and $2$ respectively and that is what we can in fact observe in the same table.  We can see that, although the local order is a bit more clear with the suggested technique, the size of the errors is slightly bigger with the suggested approach. Therefore, it seems that, in this particular problem, the error constants are bigger with the suggested technique and it is not worth the additional cost of calculating terms which contain $\varphi_3(k A_{h,0})$ and $\varphi_4(k A_{h,0})$.

The comparison is more advantageous for the suggested technique when the solution is such that it does not vanish at the boundary. Then, Theorem \ref{teorcanvbcsp} and Remark \ref{norma2} state that the local and global error should show order $2$ with the classical approach and that can be checked in Table \ref{t2}. However, with the suggested strategy, as with the vanishing boundary conditions case, the theorems in this paper prove local order $3$ and global order $2$, which can again be checked in the same table. The fact that we manage to increase the order in the local error makes that the global errors, although always of order $2$, are smaller with the suggested technique than with the classical approach. Nevertheless, the comparison between both techniques will be more beneficial for the technique which is suggested in the paper when the classical (non-exponential) order of the quadrature rule increases.

\begin{table}[t]
\begin{center}
\begin{tabular}{|c|c|c|c|c|c|c|c|c|}
\hline & \multicolumn{4}{|c|}{Classical approach} & \multicolumn{4}{|c|}{Suggested approach} \\ \hline
 k &  Loc. err. &  ord. & Glob. err. &  ord. & Loc. err. &  ord. & Glob. err. &  ord. \\\hline
 1/10 & 8.0170e-5 & & 5.5395e-5 & & 1.5334e-4 & &9.8091e-5 &   \\ \hline
 1/20 & 1.2961e-5 & 2.6 & 1.3953e-5 & 2.0 & 1.9139e-5 & 3.0 & 1.8575e-5& 2.4\\ \hline
 1/40 & 1.8644e-6 & 2.8 & 3.4952e-6 & 2.0 & 2.3902e-6 & 3.0 & 4.0262e-6 & 2.2 \\ \hline
 1/80 & 2.5316e-7 & 2.9 & 8.7426e-7 & 2.0 & 2.9862e-7 & 3.0 & 9.3773e-7 &  2.1\\ \hline
 1/160 & 3.3354e-8 & 2.9 & 2.1860e-7 & 2.0 & 3.7318e-8 & 3.0 & 2.2635e-7 &  2.0 \\ \hline
 1/320 & 4.3171e-9 & 3.0 & 5.4651e-8 & 2.0 & 4.6641e-9 & 3.0 & 5.5610e-8 &  2.0\\ \hline
\end{tabular}
\caption{Trapezoidal rule, $h=10^{-3}$,  $u(x,t)=x(1-x) e^{-t}$}
\end{center}
\label{t1}
\end{table}

\begin{table}[t]
\begin{center}
\begin{tabular}{|c|c|c|c|c|c|c|c|c|}
\hline & \multicolumn{4}{|c|}{Classical approach} & \multicolumn{4}{|c|}{Suggested approach} \\ \hline
 k &  Loc. err. &  ord.& Glob. err. &  ord. & Loc. err. &  ord. & Glob. err. &  ord. \\\hline
 1/10 &  7.4531e-4 & & 4.8108e-4 & &  2.0144e-4 &  & 1.3742e-4 & \\ \hline
 1/20 & 1.5446e-4 & 2.3 & 1.2476e-4 & 2.0 & 2.8775e-5 & 2.8 & 3.0165e-5 & 2.2\\ \hline
1/40 &  3.3863e-5 & 2.2 & 3.2074e-5 &  2.0 & 3.9084e-6 & 2.9 & 7.0453e-6 & 2.1 \\ \hline
1/80 & 7.3906e-6 & 2.2 & 8.1809e-6 & 2.0 & 5.1475e-7 & 2.9 & 1.6979e-6 & 2.0 \\ \hline
1/160 & 1.5848e-6 & 2.2 & 2.0770e-6 & 2.0 &  6.6122e-8 & 3.0 & 4.1324e-7 & 2.0 \\ \hline
1/320 & 3.3666e-7 & 2.2 & 5.2777e-7 & 2.0 &  8.1531e-9 & 3.0 & 9.8459e-8 & 2.1\\ \hline
\end{tabular}
\caption{Trapezoidal rule, $h=10^{-3}$,  $u(x,t)=e^{x-t}$}
\end{center}
\label{t2}
\end{table}

\subsection{Simpson rule}

In this subsection we consider Simpson rule in time and a collocation spectral method in space with $40$ nodes so that the error in space is negligible. As Simpson rule corresponds to $s=3$ and the interpolatory quadrature rule which is based in those $3$ nodes is exact for polynomials of degree $\le 3$, we can take $p=4$ in Theorem \ref{locps1} and achieve orders $5$ and $4$ for the local and global error respectively with the technique suggested here. However, with the classical approach, at least in the common case that $g(t)\not\equiv 0$, Theorem \ref{teorcanvbcsp}  and Remark \ref{norma2} give just order $3$ for the local and global error. These results can be checked  in Table \ref{t3}. Moreover, the size of the global error, even for the bigger timestepsizes is smaller with the suggested technique.

We also want to remark here that the trapezoidal and Simpson rules correspond to Gauss-Lobatto quadrature rules with $s=2$ and $s=3$ respectively and therefore Corollary \ref{corol} (ii) and Remark \ref{gl} apply.

\begin{table}[t]
\begin{center}
\begin{tabular}{|c|c|c|c|c|c|c|c|c|}
\hline & \multicolumn{4}{|c|}{Classical approach} & \multicolumn{4}{|c|}{Suggested approach} \\ \hline
 k &  Loc. err. &  ord. & Glob. err. &  ord. & Loc. err. &  ord. & Glob. err. &  ord. \\\hline
1/2 & 8.0718e-4 & & 4.9507e-4 & & 2.7496e-4 & &1.6821e-4 & \\ \hline
1/4 & 8.3265e-5 & 3.3 & 4.2862e-5 & 3.5 &8.5778e-6  & 5.0 & 4.4156e-6  & 5.2 \\ \hline
1/8 & 8.6214e-6 & 3.3 & 4.0561e-6  & 3.4 & 2.6785e-7  & 5.0 & 1.5345e-7  & 4.8 \\\hline
1/16 & 1.0500e-6 & 3.0 & 4.4103e-7  & 3.2 & 8.3681e-9  &5.0 & 6.7803e-9  & 4.5 \\ \hline
 1/32 & 1.1622e-7   & 3.2 &  4.6864e-8  & 3.2 &2.6148e-10  &5.0 & 3.5342e-10  & 4.3 \\ \hline
 1/64 & 1.2378e-8 &   3.2 & 4.9423e-9 & 3.2 & 8.1711e-12& 5.0 & 2.0189e-11 & 4.1\\ \hline
 \end{tabular}
\caption{Simpson's rule,  $J=39$, $u(x,t)=e^{x-t}$}
\end{center}
\label{t3}
\end{table}

\subsection{Gaussian rules}

In order to achieve the highest accuracy given a certain number of nodes, we consider in this subsection Gaussian quadrature rules. More precisely, those corresponding to $s=1,2,3,4$. As space discretization, we have considered again the same spectral collocation method of the previous subsection. Following Corollary \ref{corol} (i) and Theorem \ref{theolocalerrorfull}, even for non-vanishing boundary conditions, taking $p=2s$ in (\ref{Vh0}) and (\ref{Vhij}) the local error in time should show order $2s+1$  and the global error, using Theorem \ref{theoglobalerrorfull}, order $2s$. This should be compared with the order
$s+1$ which is proved for the classical approach when  $g(t)\equiv 0$ in Theorem \ref{teorcavbcsp} and the order
 $s$ for the local and global error when $g(t)\not \equiv 0$, which comes from Theorem \ref{teorcanvbcsp} and Remark \ref{norma2}.
 In Tables \ref{t4} and \ref{t5} we see the results which correspond to $s=1$ and $s=2$ respectively for the vanishing boundary conditions case. Although for $s=1$ there is not an improvement on the global order for the suggested technique, the errors are a bit smaller. Of course the benefits are more evident with $s=2$. For the non-vanishing boundary conditions case, Tables \ref{t6},\ref{t7},\ref{t8} and \ref{t9} show the results which correspond to $s=1,2,3$ and $4$ respectively. When avoiding order reduction, the results are much better than with the classical approach. Not only the order is bigger but also the size of the errors is smaller from the very beginning. We notice that the global order is even a bit better than expected for the first values of $k$.

\begin{table}[t]
\begin{center}
\begin{tabular}{|c|c|c|c|c|c|c|c|c|}
\hline & \multicolumn{4}{|c|}{Classical approach} & \multicolumn{4}{|c|}{Suggested approach} \\ \hline
 k &  Loc. err. &  ord. & Glob. err. &  ord. & Loc. err. &  ord. & Glob. err. &  ord. \\\hline
 1/4 & 8.2639e-3 &  & 4.3743e-3 &  & 4.9483e-3 &  & 2.5685e-3 &  \\ \hline
  1/8 & 1.8874e-3 & 2.1 & 1.1548e-3 & 1.9 & 6.4427e-4 & 2.9 & 3.7820e-4 & 2.8 \\ \hline
   1/16 & 3.6716e-4 & 2.4 & 2.9791e-4 & 2.0 & 8.2199e-5 & 3.0 & 6.9202e-5 & 2.4 \\ \hline
    1/32 & 7.6916e-5 & 2.2 & 7.6389e-5 &  2.0 & 1.0381e-5 & 3.0 & 1.4677e-5 & 2.2 \\ \hline
     1/64 & 1.6803e-5 & 2.2 & 1.9445e-5 & 2.0 & 1.3043e-6 & 3.0 & 3.3796e-6 & 2.1 \\ \hline
      1/128 & 3.6111e-6 & 2.2 & 4.9232e-6 & 2.0 & 1.6345e-7 & 3.0 & 8.1115e-7 & 2.1
      \\ \hline
\end{tabular}
\caption{Midpoint rule, $J=39$, $u(x,t)=x(1-x)e^{-t}$}
\end{center}
\label{t4}
\end{table}

\begin{table}[t]
\begin{center}
\begin{tabular}{|c|c|c|c|c|c|c|c|c|}
\hline & \multicolumn{4}{|c|}{Classical approach} & \multicolumn{4}{|c|}{Suggested approach} \\ \hline
 k &  Loc. err. &  ord. & Glob. err. &  ord. & Loc. err. &  ord. & Glob. err. &  ord. \\\hline
 1/2 &  8.9292e-4 & & 5.4788e-4 & &  8.2591e-5 & & 5.0573e-5 & \\\hline
 1/4 & 8.6746e-5 & 3.4 & 4.5296e-5 & 3.6 & 2.8465e-6 & 4.9 & 1.4782e-6 & 5.1 \\\hline
 1/8 & 8.1186e-6 & 3.4 & 3.9933e-6 & 3.5 & 9.3457e-8 & 4.9 & 5.4928e-8 & 4.7   \\\hline
 1/16 & 1.0237e-6 & 3.0 & 4.3449e-7 & 3.2 & 2.9938e-9 & 5.0 & 2.5254e-9 & 4.4 \\\hline
 1/32 & 1.1415e-7 & 3.2 & 4.6048e-8 & 3.2 & 9.4726e-11 & 5.0 &1.3424e-10 & 4.2 \\ \hline
 1/64 & 1.2112e-8 & 3.2 & 4.8418e-9 & 3.2 & 2.9786e-12 & 5.0 & 7.7190e-12 & 4.1 \\ \hline
\end{tabular}
\caption{Gaussian rule with $s=2$, $J=39$, $u(x,t)=x(1-x)e^{-t}$}
\end{center}
\label{t5}
\end{table}

\begin{table}[t]
\begin{center}
\begin{tabular}{|c|c|c|c|c|c|c|c|c|}
\hline & \multicolumn{4}{|c|}{Classical approach} & \multicolumn{4}{|c|}{Suggested approach} \\ \hline
 k &  Loc. err. &  ord. & Glob. err. &  ord. & Loc. err. &  ord. & Glob. err. &  ord. \\\hline
1/8 & 6.6985e-2  & &  2.8814e-2  & & 1.5650e-3 &  & 8.6254e-4 & \\ \hline
 1/16 & 3.0444e-2 & 1.1 & 1.2167e-2      & 1.2 & 1.8673e-4  & 3.1 & 1.4048e-4 & 2.6 \\ \hline
 1/32 & 1.3218e-2 & 1.2 & 5.1128e-3      & 1.2 & 2.2356e-5   & 3.1 & 2.7661e-5 & 2.3 \\ \hline
 1/64 & 5.6269e-3 &  1.2 & 2.1472e-3 & 1.2 & 2.6947e-6 & 3.1 & 6.1319e-6 & 2.2 \\ \hline
 1/128 & 2.3791e-3 &   1.2 & 9.0138e-4 &  1.2 & 3.2718e-7  & 3.0 & 1.4450e-6 & 2.1 \\ \hline
 1/256 & 1.0015e-3 &  1.2 & 3.7813e-4  &  1.2 & 3.9993e-8 & 3.0 & 3.5085e-7 & 2.0\\ \hline
 \end{tabular}
\caption{Midpoint rule, $J=39$, $u(x,t)=e^{x-t}$}
\end{center}
\label{t6}
\end{table}

\begin{table}[t]
\begin{center}
\begin{tabular}{|c|c|c|c|c|c|c|c|c|}
\hline & \multicolumn{4}{|c|}{Classical approach} & \multicolumn{4}{|c|}{Suggested approach} \\ \hline
 k &  Loc. err. &  ord. & Glob. err. &  ord. & Loc. err. &  ord. & Glob. err. &  ord. \\\hline
1/2 & 1.9328e-2 & & 1.1743e-2 & & 7.9408e-4 & & 4.8503e-4 & \\\hline
1/4 & 4.8715e-3   & 2.0 & 2.3068e-3 & 2.3 & 2.2565e-5   & 5.1 & 1.1356e-5 & 5.4 \\\hline
 1/8 & 1.1460e-3  & 2.1 & 4.7811e-4 & 2.3 & 6.3799e-7   & 5.1 & 3.3399e-7 & 5.1 \\\hline
  1/16 & 2.5199e-4 & 2.2 & 9.8750e-5 & 2.3 & 1.8167e-8   & 5.1 & 1.2423ee-8 & 4.7 \\\hline
   1/32 & 5.4061e-5  & 2.2 & 2.0543e-5 & 2.3 & 5.2222e-10   & 5.1 &  5.7608e-10 & 4.4 \\\hline
    1/64 & 1.1476e-5  & 2.2 &  4.2930e-6 & 2.3 & 1.4918e-11  & 5.1 &  2.9580e-11 & 4.3\\\hline
\end{tabular}
\caption{Gaussian rule with $s=2$, $J=39$, $u(x,t)=e^{x-t}$}
 \end{center}
\label{t7}
\end{table}

\begin{table}[t]
\begin{center}
\begin{tabular}{|c|c|c|c|c|c|c|c|c|}
\hline & \multicolumn{4}{|c|}{Classical approach} & \multicolumn{4}{|c|}{Suggested approach} \\ \hline
 k &  Loc. err. &  ord. & Glob. err. &  ord. & Loc. err. &  ord. & Glob. err. &  ord. \\\hline
1/2 & 8.4732e-4 & & 5.1405e-4 & & 3.0107e-6 & & 1.8363e-6 & \\ \hline
1/4 &  1.0734e-4   & 3.0 & 5.0710e-5  & 3.3 & 2.0423e-8    & 7.2 & 1.0082e-8  & 7.5  \\ \hline
1/8 & 1.2260 e-5  & 3.1 & 5.1107e-6 & 3.3 & 1.3840e-10  & 7.2 & 6.6752e-11 & 7.2 \\ \hline
1/16 & 1.3379e-6  & 3.2 & 5.2398e-7  & 3.3 & 9.4087e-13  & 7.2 & 5.3198e-13 & 7.0 \\ \hline
1/32 & 1.4335e-7 & 3.2 & 5.4413e-8  & 3.3 & 6.4275e-15  & 7.2 & 5.3534e-15  & 6.6 \\ \hline
1/64 & 1.5182e-8 & 3.2 & 5.6735e-9 &  3.3 & 4.4259e-17 & 7.2 & 6.6517e-17 & 6.3 \\ \hline
\end{tabular}
\caption{Gaussian rule with $s=3$, $J=39$, $u(x,t)=e^{x-t}$}
 \end{center}
\label{t8}
\end{table}

\begin{table}[t]
\begin{center}
\begin{tabular}{|c|c|c|c|c|c|c|c|c|}
\hline & \multicolumn{4}{|c|}{Classical approach} & \multicolumn{4}{|c|}{Suggested approach} \\ \hline
 k &  Loc. err. &  ord. & Glob. err. &  ord. & Loc. err. &  ord. & Glob. err. &  ord. \\\hline
 1/2 & 2.7746e-5 & & 1.6829e-5 & & 8.1253e-9 & & 4.9495e-9 &  \\\hline
  1/4 & 1.7137e-6 & 4.0 & 8.0951e-7  & 4.4 & 1.3502e-11  & 9.2 & 6.5604e-12  & 9.6 \\\hline
   1/8 & 9.6826e-8  & 4.1 & 4.0363e-8 & 4.3 & 2.2443e-14  & 9.2 & 1.0107e-14  & 9.3  \\\hline
    1/16 & 5.2833e-9 & 4.2 & 2.0690e-9  & 4.3 & 3.7267e-17  &  9.2 & 1.7380e-17 & 9.2 \\\hline
     1/32 &  2.8240e-10 & 4.2 & 1.0719e-10 & 4.3 & 6.2266e-20 & 9.2 & 3.5644e-20 & 8.9 \\\hline
\end{tabular}
\caption{Gaussian rule with $s=4$, $J=39$,  $u(x,t)=e^{x-t}$}
 \end{center}
\label{t9}
\end{table}

Finally, although it is not an aim of this paper, in order to compare roughly the results in terms of computational cost, let us concentrate on Gaussian quadrature rules of the same order $2s$ when integrating a non-vanishing boundary value problem. When considering $2s$ nodes with the classical approach, $2s$ evaluations of the source term $f$ must be made at each step and the $2s$ operators $\{ \varphi_j(k A_{h,0}) \}_{j=1}^{2s}$ are needed, which will be multiplied by vectors with all its components varying in principle at each step. However, with the suggested technique and $s$ nodes, just $s$ evaluations of the source term $f$ must be made although $3s$ operators $\{\varphi_j(k A_{h,0})\}_{j=1}^{3s}$ are needed. Nevertheless, from these $3s$, just the first $s$ of them are multiplied by vectors which change independently in all their components at each step. The other $2s$ are multiplied by vectors which just contain information on the boundary. Therefore, with finite differences many components vanish and, with Gauss-Lobatto spectral methods, those vectors are just a time-dependent linear combination of two vectors which do not change with time. With Gauss-Lobatto methods, as $A_{h,0}$ is not sparse but its size its moderate, we have calculated once and for all at the very beginning $e^{k A_{h,0}}$, $\varphi_j(k A_{h,0})$, $j=1,\dots,s$ and the two necessary vectors derived from $\varphi_j(k A_{h,0})$, $j=s+1,\dots,3s$. Then, in (\ref{ff}) the terms containing the former at each step require $O(J^2)$ operations while the terms containing the latter just require $O(J)$ operations. With finite differences, as the matrix $A_{h,0}$ is sparse and usually bigger, we have applied general Krylov subroutines \cite{niesen} to calculate all the required terms at each step.

We offer a particular comparison for order $2$ with Gauss-Lobatto spectral space discretization on the one hand and 2nd-order finite differences on the other, and considering the implementation described above in each case. In Figure \ref{f1} we can see that, for the former, the suggested technique is more than twice cheaper than the classical one and, with the latter in Figure \ref{f2}, the comparison is not so advantageous for the suggested technique but it is still cheaper than the classical approach. We also remark that, in any case, the more expensive the source function $f$ is to evaluate, the more advantageous the suggested technique with $s$ nodes will be against the classical approach with $2s$ nodes.

\begin{figure}
\begin{center}
\epsfig{file=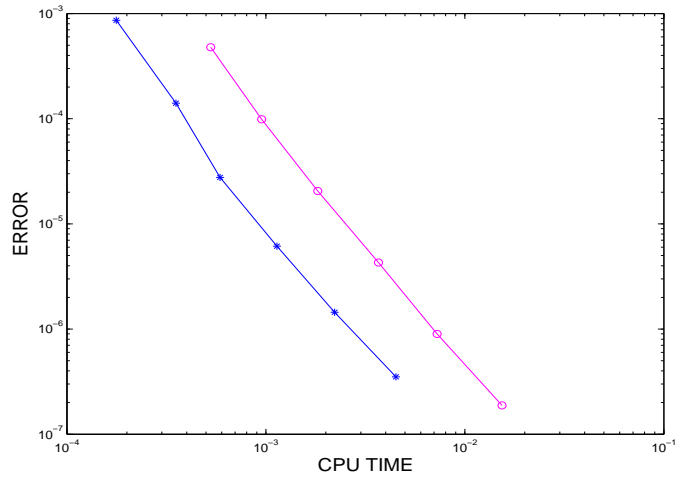,height=2.5in,width=3.5in}
\end{center}
\caption{Error against CPU time when integrating problem (\ref{pr}) with exact solution $u(x,t)=e^{x-t}$, using Gauss-Lobatto spectral method in space and, in time,
the classical approach of Gaussian rule with $s=2$ (pink circles) or the suggested technique for midpoint rule (blue asterisks)  }
\label{f1}
\end{figure}

\begin{figure}
\begin{center}
\epsfig{file=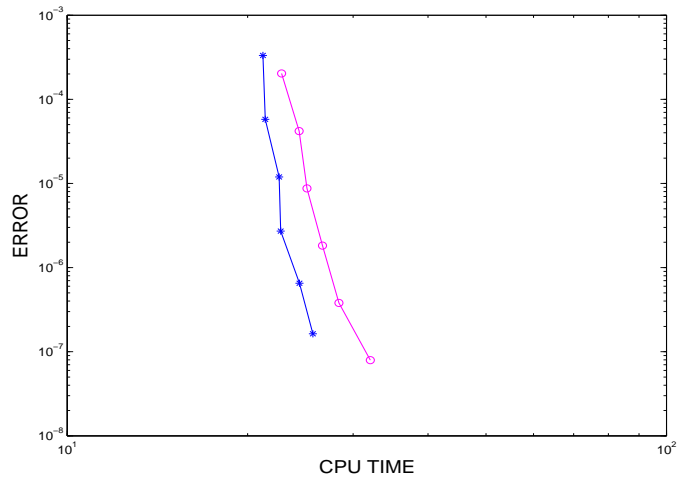,height=2.5in,width=3.5in}
\end{center}
\caption{Error against CPU time when integrating problem (\ref{pr}) with exact solution $u(x,t)=e^{x-t}$, using second-order finite differences in space and, in time,
the classical approach of Gaussian rule with $s=2$ (pink circles) or the suggested technique for midpoint rule (blue asterisks)  }
\label{f2}
\end{figure}



\bibliographystyle{siamplain}

\end{document}